\documentclass[a4paper,11pt,oneside]{amsart}
\usepackage{amsmath}
\usepackage{amsthm}
\usepackage{amssymb}
\usepackage[T1]{fontenc}
\usepackage{dsfont}
\usepackage[many]{tcolorbox}
\usepackage{soul}
\usepackage{xcolor}
\usepackage{graphicx}
\usepackage{adjustbox}
\usepackage{tikz-cd}
\usepackage{appendix}
\usepackage{enumerate}
\usepackage{multicol}
\usepackage{fancyhdr}
\usepackage{parskip}
\usepackage{hyperref}
\usepackage{cleveref}
\usepackage{mathtools}
\usepackage{nameref}
\usepackage{thmtools}
\usepackage{thm-restate}
\usepackage{setspace}
\usepackage{float}
\usepackage{bm}
\usepackage{orcidlink}
\usepackage{braket}
\usepackage{calc}

\hypersetup{
  colorlinks   = true, 
  urlcolor     = {blue!50!black}, 
  linkcolor    = {blue!50!black}, 
  citecolor   = {red!50!black} 
}

\usepackage{ulem}
\usepackage{todonotes}

\colorlet{NadavColor}{-green!40!yellow}
\colorlet{ArisColor}{red}

\newtheorem{theorem}{Theorem}[section]
\newtheorem*{theorem*}{Theorem}
\newtheorem{lemma}[theorem]{Lemma}
\newtheorem{proposition}[theorem]{Proposition}
\newtheorem*{proposition*}{Proposition}
\newtheorem{corollary}[theorem]{Corollary}
\newtheorem*{corollary*}{Corollary}
\newtheorem*{fact*}{Fact}
\newtheorem{fact}[theorem]{Fact}

\newtheorem{thmx}{Theorem}

\theoremstyle{definition}
\newtheorem*{definition*}{Definition}
\newtheorem{definition}[theorem]{Definition}

\newtheorem*{question*}{Question}
\newtheorem{example}[theorem]{Example}

\theoremstyle{remark}
\newtheorem{remark}[theorem]{Remark}
\newtheorem{claim}{Claim}

\newtheorem{poc}{Proof of Claim}

\newenvironment{claimproof}{%
	\begin{poc}
	}{%
	\hfill$\blacktriangleleft$
	\end{poc}
	}

\newcommand*{\Scale}[2][4]{\scalebox{#1}{\ensuremath{#2}}}

\newcommand{\into}{\ensuremath{\hookrightarrow}}

\def\Ind#1#2{#1\setbox0=\hbox{$#1x$}\kern\wd0\hbox to 0pt{\hss$#1\mid$\hss}
\lower.9\ht0\hbox to 0pt{\hss$#1\smile$\hss}\kern\wd0}

\def\notind#1#2{#1\setbox0=\hbox{$#1x$}\kern\wd0
\hbox to 0pt{\mathchardef\nn=12854\hss$#1\nn$\kern1.4\wd0\hss}
\hbox to 0pt{\hss$#1\mid$\hss}\lower.9\ht0 \hbox to 0pt{\hss$#1\smile$\hss}\kern\wd0}

\newcommand{\N}{\mathbb{N}}

\newcommand{\M}{\mathcal{M}}
\newcommand{\C}{\mathcal{C}}
\renewcommand{\L}{\mathcal{L}}

\newcommand{\cL}{\mathcal{L}}
\newcommand{\cM}{\mathcal{M}}

\newcommand{\cQ}{\mathcal{Q}}

\newcommand{\Flim}{\ensuremath{\mathsf{Flim}}}

\newcommand{\restr}{\upharpoonright}

\newcommand{\iso}{\simeq}

\DeclareMathOperator{\Age}{\mathsf{Age}}

\DeclareMathOperator{\Diag}{\mathsf{Diag}}

\DeclareMathOperator{\Sym}{\mathsf{Sym}}

\let\aut\Aut
\let\vphi\varphi

\newcommand{\Ccal}{\ensuremath{\mathcal{C}}}

\newcommand{\Hcal}{\ensuremath{\mathcal{H}}}

\newcommand{\Mcal}{\ensuremath{\mathcal{M}}}

\newcommand{\Nbb}{\ensuremath{\mathbb{N}}}

\renewcommand{\Im}{\mathsf{Im}}

\title{All These Approximate Ramsey Properties}
\date{\today}



\author[N. Meir]{Nadav Meir \orcidlink{0000-0002-7774-2892}}
\address{Nadav Meir, Instytut Matematyczny, Uniwersytet Wroc{\l}awski, pl. Grunwaldzki 2, 50-384 Wroc{\l}aw, Poland}

\email{\href{mailto:nadavmeir@gmail.com}{nadavmeir@gmail.com}}
\author[A. Papadopoulos]{Aris Papadopoulos \orcidlink{0000-0001-7071-4277}}
\thanks{The first author is supported by Narodowe Centrum Nauki, Poland, grant 2016/22/E/ST1/00450. This research is part of the second author’s Ph.D. project, supported by a Leeds Doctoral Scholarship, from the University of Leeds.}

\address{Aris Papadopoulos, School of Mathematics, University of Leeds, Leeds LS2 9JT, United Kingdom}
\email{\href{mailto:mmadp@leeds.ac.uk}{mmadp@leeds.ac.uk}}

\begin{document}

\maketitle

\begin{abstract}
    We consider finitary approximations of the (embedding) Ramsey property. Using a class of homogeneous reducts of random ordered hypergraphs, we prove that these properties form a strict hierarchy. We also show that every class of finite structures in which every structure of size at most $2$ is a ``Ramsey object'' essentially consists of ordered structures, generalising a known result for countable Ramsey classes.
\end{abstract}

\section{Introduction}

A class $\Ccal$ of finite structures is called a \emph{Ramsey class} if it satisfies a structural analogue of Ramsey's theorem. More precisely $\Ccal$ is a Ramsey class if for every structure $A\in\Ccal$ the following holds:
\begin{itemize}
    \item[($\diamond_A$)] For all $B\in\Ccal$ such that $A\subseteq B$, there exists some $C\in\Ccal$ such that for all $n$-colourings $\chi$ of the substructures of $C$ isomorphic to $A$, there is a copy of $B$ in $C$ on which $\chi$ is constant.
\end{itemize}
We will discuss this in more detail in \Cref{sec:background}. With this definition in mind, Ramsey's theorem \cite{Ramsey_1930} says that the class of all finite linear orders is a Ramsey class. 

Structural Ramsey theory (the study of Ramsey classes) has been a very active topic of research in the last twenty years, especially after the beautiful results of \cite{KechrisPestovTodorcevic_2005} which drew deep connections between Ramsey classes and extremely amenable (topological) groups.

Let us now briefly introduce our main definition. We will say that $\Ccal$ is a \emph{$k$-Ramsey class} or that $\Ccal$ is \emph{Ramsey up to $k$} if $(\diamond_A)$ holds for all $A\in\Ccal$ with $|A|\leq k$. The following assertions are immediate:

\begin{enumerate}
    \item If $\Ccal$ is a Ramsey class, then it is a $k$-Ramsey class, for all $k\in\Nbb_{\geq 1}$.
    \item\label{item:obvious-implication} For all $k\in\Nbb_{\geq 1}$, if $\Ccal$ is Ramsey up to $k+1$, then it is Ramsey up to $k$. 
\end{enumerate}

So we may naturally view Ramsey up to $k$ as a finite approximation of being Ramsey. Our first goal is to prove that the implications in (\labelcref{item:obvious-implication}), above, are all strict. More precisely, we show the following:

\begin{thmx}[\Cref{cor:main}]\label{thmx:strict}
    For all $k\in\Nbb_{\geq 1}$
    there is a class $\Ccal$ of finite ordered structures which is Ramsey up to $k$, but not up to $k+1$.
\end{thmx} 

We note that different authors use the term \emph{Ramsey class} to refer to two different notions,  \emph{structural} Ramsey classes or \emph{embedding} Ramsey classes --- the difference being whether the objects that are coloured are ``isomorphic substructures'' or ``embeddings'' (see \Cref{sec:background} for the precise definitions). However, these two notions coincide for classes of ordered structures, so in the \namecref{thmx:strict} above we capture both.

We then turn our attention to linear orders in Ramsey classes. It is a well-known fact (see \cite[Proposition 2.23]{Bodirsky_2015}) that if $\Mcal$ is a (countable) ultrahomogeneous structure whose age is a Ramsey class, then there exists a linear order on the domain of $\Mcal$ which is preserved by all automorphisms of $\Mcal$. We generalise this theorem to obtain the following:

\begin{thmx}[\Cref{thm:2-ERP-implies-order}]\label{thmx:order}
    Let $\Ccal$ be a class of finite structure. If $\Ccal$ is embedding Ramsey up to $2$ (see \Cref{def:erp}) then there is a possibly infinite Boolean combination of atomic formulas $\Phi(x,y)$ which defines a linear order on all structures in $\Ccal$.
\end{thmx}

We dub the condition in the conclusion of \Cref{thmx:order} \emph{$(\infty,0)$-orderability} (see \Cref{def:orderability}). As observed in \cite{MP}, $(\infty,0)$-orderability is a naturally arising condition in various model-theoretic contexts. We also observe that classes of finite structures which are $(\infty,0)$-orderable must necessarily consist of rigid structures, but not all classes of finite rigid structures are $(\infty,0)$-orderable (see \Cref{example:rigid not ordered finite language}).

\subsection*{Structure of the paper} 
For the sake of keeping this paper self-contained, we start by recalling all the standard definitions that will be used in our paper in \Cref{sec:background}. Then, in \Cref{sec:kay-graphs} we introduce \emph{\emph{Kay}-graphs}, which generalise \emph{two}-graphs to higher arities and discuss some of their properties. In \Cref{sec:approximate} we use \emph{Kay}-graphs to prove \Cref{thmx:strict}. Finally, in \Cref{sec:elem-ord} we prove \Cref{thmx:order}.

\subsection*{Notation and conventions} Most of the techniques used in our paper, with the exception of \Cref{sec:elem-ord}, are elementary. For \Cref{sec:elem-ord}, we assume some familiarity with basic model-theoretic notions such as \emph{compactness}, \emph{quantifier-free types} and \emph{diagrams}. An excellent reference for these notions is \cite{TentZiegler_2012}. We take the opportunity here to summarise the notational conventions that will be used throughout the paper. 

Given finite structures $A$ and $B$, we denote by $\binom{B}{A}$ the set of all substructures of $B$ which are isomorphic to $A$ and by $\mathsf{Emb}(A,B)$ the set of all embeddings of $A$ into $B$.

Let $k\in\Nbb_{\geq 2}$. By a $k$-hypergraph we mean a structure $H=(V;R)$, where $R$ is a $k$-ary relation symbol such that for all $x_1,\dots,x_k\in V$ we have that if $(x_1,\dots,x_k)\in R$ then $|\{x_1,\dots,x_k\}| = k$; and if $(x_1,\dots,x_k)\in R$ then, for all $\sigma\in\Sym(k)$ we also have $\left(x_{\sigma(1)},\dots,x_{\sigma(k)}\right)\in R$. More concisely, $H=(V,R)$ is a $k$-hypergraph, then $R\subseteq\binom{V}{k}$, where $\binom{V}{k}$ denotes the collection of all $k$-element subsets of $V$. We will use the notation $(v_1,\dots,v_k)\in R$ and $\{v_1,\dots,v_k\}\in R$ interchangeably (and to the same effect). If $\L$ is a relational language, $H$ an $\L$-structure and $R\in\L$ we will sometimes write $R(H)$ (or even just $R$ if $H$ is understood) for the set of realisations of $R$ in $H$.

\section{Preliminaries}\label{sec:background}

Let us start by formally defining some terminology that we mentioned in the introduction.

\begin{definition}\label{def:Ramsey degree}
    Let $\Ccal$ be a class of finite structures, and $k\in\Nbb_{\geq 1}$. We say that $A\in\Ccal$ has \emph{structural (\emph{resp.} embedding) Ramsey degree at most $k$} if for every $B\in\Ccal$ and every $n\in\Nbb$ there is some $C\in\Ccal$ such that for every colouring $\chi: \binom{C}{A}\to[n]$ (resp. $\chi:\mathsf{Emb}(A,C)\to[n]$) there is $B' \in \binom{C}{B}$ (resp. $B'\in\mathsf{Emb}(B,C)$) with $|\mathsf{Im}\left(\chi_{\restr\binom{B'}{A}}\right)|\leq k$. 
\end{definition}

Recall the following well-known fact connecting structural and embedding Ramsey degrees:

\begin{fact}[{\cite[Corollary~4.5]{Zucker2015}}]\label{fact:Andys equality}
Let $\Ccal$ be a class of finite structures, and $A\in\Ccal$. Then, $A$ has \emph{structural} Ramsey degree $k$ if, and only if, $A$ has \emph{embedding} Ramsey degree $k |\mathsf{Aut}(A)|$.
\end{fact}

We call $A\in\Ccal$ a \emph{Ramsey object} (for $\Ccal$) if it has embedding Ramsey degree at most $1$. Thus, $A$ is a Ramsey object for $\Ccal$ if, and only if it is rigid\footnote{Recall: A structure is \emph{rigid} if it has no non-trivial automorphisms.} and has structural Ramsey degree $1$.

\begin{definition}\label{def:erp}
    We say that a class of finite structures has the \emph{Embedding Ramsey Property} (\emph{ERP}) if every $A\in\Ccal$ is a Ramsey object. For $k\in\Nbb_{\geq 1}$ we say that $\Ccal$ has the \emph{$k$-Embedding Ramsey Property} (\emph{$k$-ERP}) if every structure of size $k$ in $\Ccal$ is a Ramsey object.
\end{definition}

\begin{lemma}\label{Ramsey-many-degs}
    Let $\Ccal$ be a class of finite structures. Let $A_1,\dots A_n\in \Ccal$ be structures contained in some $B\in\Ccal$, with $A_i$ having (embedding) Ramsey degree $d_i$, for all $i\leq n$. Then, there is some $C\in\Ccal$ such that $B\subseteq C$ and for all $r_i\in\N$ and all sets of colourings $\{\chi_i:\binom{C}{A_i}\to r_i:i\leq n\}$, there is some $\tilde B\in\binom{C}{B}$ such that $|\Im({\chi_i}_{\restr{\binom{\tilde B}{A_i}}})|\leq d_i$, for all $i\leq n$.
\end{lemma}

\begin{proof}
    The proof is by induction on $n$. The base case is precisely the definition of $A_1$ having Ramsey degree $d_1$ in $\Ccal$. For the inductive step, assume that we are given $A_1,\dots,A_{n+1},B\in\Ccal$ such that $A_i\subseteq B$ and the Ramsey degree of $A_i$ is $d_i$, for all $i\leq n+1$. By the inductive hypothesis, there is some $C_0\in\Ccal$ such that for all $r_i\in\N$, $i\leq n$, and all sets of colourings $\{\chi_i:\binom{C_0}{A_i}\to r_i:i\leq n\}$, there is some $\tilde B\in\binom{C_0}{B}$ such that $|\Im({\chi_i}_{\restr{\binom{\tilde B}{A_i}}})|\leq d_i$, for all $i\leq n$. Now, since $A_{n+1}$ has Ramsey degree $d_{n+1}$ in $\C$, there is some $C\in\C$ such that for all $r\in\N$ and all colourings $\chi: \binom{C}{C_0}\to r$ there exists some $\tilde C_0\in\binom{C}{C_0}$ such that $|\Im(\chi_{\restr{\binom{\tilde C_0}{A_{n+1}}}})| \leq d_{n+1}$.
    
    We claim that this $C$ is the required structure. Indeed, suppose that we are given $r_i\in\N$, $i\leq n+1$, and a set of colourings $\{\chi_i:\binom{C}{A_i}\to r_i:i\leq n+1\}$. Let $\tilde C_0\in\binom{C}{C_0}$ be such that $|\Im(\chi_{\restr{\binom{\tilde C_0}{A_{n+1}}}})| \leq d_{n+1}$. Then, by induction, there is some $\tilde B\in\binom{\tilde C_0}{B}$ such that $|\Im({\chi_i}_{\restr{\binom{\tilde B}{A_i}}})|\leq d_i$, for all $i\leq n$. But since $\binom{\tilde B}{A_{n+1}}\subseteq\binom{\tilde C_0}{A_{n+1}}$, we must have that $\Scale[0.9]{\left|\Im\left({\chi_{n+1}}_{\restr{\binom{\tilde B}{A_{n+1}}}}\right)\right|\leq \left|\Im\left(\chi_{\restr{\binom{\tilde C_0}{A_{n+1}}}}\right)\right| \leq d_{n+1}}$, as required.
\end{proof}

As an immediate corollary, we obtain the following useful fact:

\begin{corollary}\label{k-ERP-many}
    Let $\Ccal$ be a class of finite structures and assume that $\Ccal$ has $k$-ERP. Then for all $i\leq n$, and $A_1,\dots A_n\in \C$ with $|A_i| = k$ and all $B\in\C$ such that $A_i\subseteq B$ for each $i\leq n$, there is some $C\in\C$ such that for all sets of colourings $\{\chi_i:\binom{C}{A_i}\to [2]:i\leq n\}$, there is some $\tilde B\in\binom{C}{B}$ such that ${\chi_i}_{\restr{\binom{\tilde B}{A_i}}}$ is constant, for all $i\leq n$.
\end{corollary}

\begin{definition}[HP, JEP, AP]\label{def:hp-jep-ap}
    Let $\L$ be a first-order language and $\C$ a class of $\L$-structures. We say that $\C$ has the:
	\begin{enumerate}[(1)]
            \item \emph{Hereditary Property} (HP) if whenever $A\in\C$ and $B\subseteq A$ we have that $B\in\C$.
            \item \emph{Joint Embedding Property} (JEP) if whenever $A,B\in\C$ there is some $C\in\C$ such that both $A$ and $B$ are embeddable in $C$.
            \item \emph{Amalgamation Property} (AP) if whenever $A,B,C\in\C$ are such that $A$ embeds into $B$ via $e:A\into B$ and into $C$ via $f:A\into C$ there exist a $D\in\C$ and embeddings $g:B\into D$, $h:C\into D$ such that $g\circ e = h\circ f$.
	\end{enumerate}
We say that $\Ccal$ is a \emph{Ramsey class} if it has HP, JEP and ERP. 
\end{definition}

\begin{fact}[\cite{Nesetril_2005}]
    If $\Ccal$ is a Ramsey class, then $\Ccal$ has AP.
\end{fact}

Recall that for a structure $\M$ we define the \emph{age} of $\M$, denoted by $\Age(\M)$, to be the class of (isomorphism types of) finitely generated substructures of $\M$. We say that $\M$ is \emph{ultrahomogeneous} if every isomorphism between finitely generated substructures of $\M$ extends to an automorphism of $\M$.

\begin{fact}[Fraïssé's Theorem]\label{fraisse}
    Let $\C$ be a countable class of finite structures with HP, JEP and AP. Then there exists a unique, up to isomorphism, countable ultrahomogeneous $\L$-structure $\M$ such that $\Age(\M) = \C$.
\end{fact}

In the notation of \Cref{fraisse}, we will write $\Flim(\C)$ to denote the unique, up to isomorphism, countable ultrahomogeneous $\M$ such that $\Age(\M) = \C$.

\section{\emph{Kay}-graphs}\label{sec:kay-graphs}
Recall that given an ordered graph $G=(V;R,\leq)$ we define the \emph{ordered two-graph} of $G$ to be the ordered $3$-hypergraph $K(G) = \left(V;R^{(3)},\leq \right)$, on the same ordered vertex set $(V;\leq)$, with hyperedge relation $R^{(3)}$ defined as follows:
\[
	R^{(3)}(x_1,x_2,x_3)
	\iff |\left\{\left\{i,j\right\}\in\binom{[3]}{2}:\{x_i,x_j\}\in R\right\}|\equiv 1\pmod{2},
\]

Various properties of two-graphs have been investigated deeply, in the last 50 years. See, for instance, \cite{SEIDEL_1991}, for an excellent survey. In this paper, we are interested in a higher-arity generalisation of two-graphs, which appears implicitly in \cite{Thomas_1996}.

We generalise this construction to higher-arity structures in the natural way, as follows:
\begin{definition}
	Given a $k$-hypergraph $H=(V;R)$ we define the \emph{\emph{Kay}-graph} of $H$ to be the $(k+1)$-hypergraph $K(H) = \left(V;R^{(k+1)}\right)$ on the same vertex set $V$ as $H$ with hyperedge relation $R^{(k+1)}$ defined as follows:
\[
\Scale[0.9]{
\begin{aligned}
		R^{(k+1)}&(x_1,\dots,x_{k+1}) \\
		&\iff |\left\{\left\{i_1,\dots,i_k\right\}\in\binom{[k+1]}{k}:\{x_{i_1},\dots,x_{i_k}\}\in R\right\}|\equiv k+1\pmod{2} 
\end{aligned}}
\]
We similarly define the \emph{ordered \emph{Kay}-graph} of an ordered hypergraph. If $\Hcal$ is a class of (ordered) $k$-hypergraphs we will write $K(\Hcal)$ for the class of all (ordered) \emph{Kay}-graphs of hypergraphs from $\Hcal$, that is $K(\Hcal):=\{K(H):H\in\Hcal\}$.
\end{definition}

In the remainder of this section, we will only discuss \emph{Kay}-graphs, but all the arguments naturally go through if one considers ordered \emph{Kay}-graphs, instead.

Let $P^{(k)}$ be the class of all finite $(k+1)$-hypergraphs $(V;R)$ satisfying the following condition:
\begin{itemize}
	\item[($\dagger$)] For all $V_0\in\binom{V}{k+2}$ the induced $(k+1)$-subhypergraph $(V_0,R_{\restr{V_0}})$ has $|R_{\restr{V_0}}|\equiv k\pmod{2}$.
\end{itemize}

\begin{proposition}\label{prop:kay-graphs-parity}
	Let $k\in\Nbb_{\geq 2}$. Then $P^{(k)} = K(\Hcal_{k})$.
\end{proposition}

\begin{proof}
	We first show the inclusion $P^{(k)}\subseteq K(\Hcal_{k})$. Let $(V;S)$ be a $(k+1)$-hypergraph and assume that $(V;S)\in P^{(k)}$. We need to show that there is a $k$-hypergraph $H=(V;R)$ such that $(V;S) = K(H)$. Fix a point $\star\in V$ and define a $k$-ary relation $R\subseteq \binom{V}{k}$ as follows:
	\[
		R(x_1,\dots,x_k)\iff 
		\begin{cases}
			S(x_1,\dots,x_k,\star), & \text{ if }\star\not\in\{x_1,\dots,x_k\} \\
			S(x_1,\dots,x_{k-1},\star) &\text{ otherwise}, 
		\end{cases}
	\]
    for all $\{x_1,\dots,x_k\}\in\binom{V}{k}$. Observe that by definition we have that:
	\begin{equation}\label{eq:parity-k-graph}
        \Scale[0.9]{
	   \begin{aligned}
	       &R^{(k+1)}(x_1,\dots,x_{k+1}) \\
	       &\iff |\left\{\left\{i_1,\dots,i_k\right\}\in\binom{[k+1]}{k}:\{x_{i_1},\dots,x_{i_k}\}\in R\right\}|\equiv k+1\pmod{2} \\
	       &\iff |\left\{\left\{i_1,\dots,i_{k}\right\}\in\binom{[k+1]}{k}:\{x_{i_1},\dots,x_{i_{k}},\star\}\in S\right\}|\equiv k+1\pmod{2}\\
	       &\iff |S_{\restr\{x_1,\dots,x_{k+1},\star\}})\setminus\{x_1,\dots,x_{k+1}\}| \equiv k+1\pmod{2},
	\end{aligned}
        }
	\end{equation}
    for all $\{x_1,\dots,x_{k}\}\in\binom{V\setminus\{\star\}}{k}$.
    \begin{claim}
        The induced \emph{Kay}-graph $\left(V;R^{(k+1}\right)$ of $(V;R)$ is precisely $(V;S)$.
    \end{claim}
    
    \begin{claimproof}
        We need to show that for all $\{x_1,\dots,x_k\}\in \binom{V}{k+1}$ we have that:
	\[
	       (x_1,\dots,x_{k+1})\in R^{(k+1)} \text{ if, and only if } (x_1,\dots,x_{k+1})\in S. 
	\]
	We consider two cases.
	
	\emph{{Case 1}.} $\{x_1,\dots,x_{k+1}\}\in \binom{V\setminus\{\star\}}{k+1}$. 
	
        First, assume that $(x_1,\dots,x_{k+1})\in R^{(k+1)}$. Since $(V;S)\in P^{(k)}$, by ($\dagger$), the induced subhypergraph $\left(\{x_1,\dots,x_{k+1},\star\};S_{\restr\{x_1,\dots,x_{k+1},\star\}}\right)$ of $(V;S)$ must have:
	\[
		|S_{\restr\{x_1,\dots,x_{k+1},\star\}}|\equiv k\pmod{2}
	\]
	Since $(x_1,\dots,x_{k+1})\in R^{(k+1)}$, we know, by \Cref{eq:parity-k-graph} that:
	\[
	       |S_{\restr\{x_1,\dots,x_{k+1},\star\}})\setminus\{x_1,\dots,x_{k+1}\}| \equiv k+1\pmod{2},
	\]
	so we must have that $\{x_1,\dots,x_{k+1}\}\in S$, as claimed. 
	
        Conversely, suppose that $(x_1,\dots,x_{k+1})\in S$. Then, by ($\dagger$) the number of hyperedges in the induced subhyphergraph $\left(\{x_1,\dots,x_{k+1},\star\};S_{\restr\{x_1,\dots,x_{k+1},\star\}}\right)$ of $(V;S)$ must be congruent to $k$, modulo $2$. In particular, if
	\[
	       |\left\{\left\{i_1,\dots,i_{k}\right\}\in\binom{[k+1]}{k}:\{x_{i_1},\dots,x_{i_{k}},\star\}\in S\right\}|\equiv k\pmod{2},
	\]
        then we cannot have that $(x_1,\dots,x_{k+1})\in S$, contrary to our original assumption. Hence
	\[
	       |\left\{\left\{i_1,\dots,i_{k}\right\}\in\binom{[k+1]}{k}:\{x_{i_1},\dots,x_{i_{k}},\star\}\in S\right\}|\equiv k+1\pmod{2},
	\]
        and thus, by \Cref{eq:parity-k-graph} we have that $(x_1,\dots,x_{k+1})\in R^{(k+1)}$, as required.
	
        \emph{Case 2.} $\{x_1,\dots,x_{k+1}\}\in \binom{V}{k+1}$, and, without loss of generality, $x_{k+1}=\star$. 
        
        Suppose first that $(x_1,\dots,x_{k},\star)\in R^{(k+1)}$. By definition of $R$ we have that $(x_{i_1},\dots,x_{i_{k-1}},\star)\in R$, for all $\{i_1,\dots,i_{k-1}\}\in \binom{[k]}{k-1}$, and thus
	\[
        \Scale[0.9]{
	\begin{aligned}
            |\bigg\{\{i_1,\dots,i_k\}\in\binom{[k+1]}{k}:\{x_{i_1},&\dots,x_{i_k}\}\in R\bigg\}| \\
            &= k+|\left\{\left\{i_1,\dots,i_k\right\}\in\binom{[k]}{k}:\{x_{i_1},\dots,x_{i_k}\}\in R\right\}| \\
            &= \begin{cases}
			k+1 &\text{ if } \{x_1,\dots,x_k\}\in R\\
			k &\text{ otherwise}.
		\end{cases}
	\end{aligned}
        }
	\]
        But, since $(x_1,\dots,x_{k},\star)\in R^{(k+1)}$, the cardinality of this set is congruent to $k+1$ modulo $2$, so we must have that  $(x_1,\dots,x_k)\in R$, and thus $(x_1,\dots,x_k,\star)\in S$, as required.
	
        Conversely, suppose that $(x_1,\dots,x_k,\star)\in S$. If $(x_1,\dots,x_{k},\star)\not\in R^{(k+1)}$, then, arguing as before we have that $(x_1,\dots,x_k)\not\in R$ and thus $(x_1,\dots,x_{k},\star)\not\in S$, contradicting our assumption. Hence we must have that $(x_1,\dots,x_{k},\star)\in R^{(k+1)}$, concluding the proof of the claim.
    \end{claimproof}
	
    We now move on to the inclusion $K(\Hcal_{k})\subseteq P^{(k)}$. We will prove by induction on the cardinality of the hyperedge set of a finite $k$-graph $H=\left(V;R\right)\in \Hcal_{k}$ that $K(H)\in P^{(k)}$.

    Let $H=\left(V;R\right)\in \Hcal_{k}$. First, note that if $R=\emptyset$ then we can easily see that  $K(H)\in P^{(k)}$. Indeed, if $k\equiv 0\pmod{2}$ then $R^{(k+1)} = \emptyset$ and hence $\left(V;R^{(k+1)}\right)\in P^{(k)}$, because $(\dagger)$ holds trivially for all $(k+2)$-element subsets of $V$. Similarly, if $k\equiv 1\pmod{2}$, then $R^{(k+1)} = \binom{V}{k+1}$, and hence $\left(V;R^{(k+1)}\right)\in P^{(k)}$, because for every $(k+2)$-element subset of $V$, the number of $(k+1)$-hyperedges is $\binom{k+2}{k+1}\equiv k+2\equiv k\pmod{2}$, and thus ($\dagger$) holds. 

    Now, assume that for any hyperedge $\{x_1,\dots,x_{k}\}\in R$ we have that $K((V;R'))\in P^{(k)}$, where $R' = R\setminus\{(x_1,\dots, x_{k})\}$. We claim that then $K((V;R))\in P^{(k)}$. 

    Indeed, let $\{y_1,\dots,y_{k+2}\}\subseteq V$. We need to show that in the induced subhypergraph $\left(\{y_1,\dots,y_{k+2}\};R^{(k+1)}_{\restr\{y_1,\dots,y_{k+2}\}}\right)$ we have that $\big| R^{(k+1)}_{\restr\{y_1,\dots,y_{k+2}\}} \big|\equiv{k}\pmod{2}$. By inductive hypothesis, $K((V;R'))\in P^{(k)}$, so this is the case for $\left(\{y_1,\dots,y_{k+2}\};{R'}^{(k+1)}_{\restr\{y_1,\dots,y_{k+2}\}}\right)$. Of course, if $\{x_1,\dots,x_k\}\not\subseteq\{y_1,\dots,y_{k+2}\}$ then $\left(\{y_1,\dots,y_{k+2}\};R^{(k+1)}_{\restr\{y_1,\dots,y_{k+2}\}}\right) = \left(\{y_1,\dots,y_{k+2}\};{R'}^{(k+1)}_{\restr\{y_1,\dots,y_{k+2}\}}\right)$, and we are done. Thus we may assume that $\{x_1,\dots,x_{k}\}\subseteq\{y_1,\dots,y_{k+2}\}$. To simplify notation, assume that $\{y_1,\dots,y_{k+2}\} = \{x_1,\dots,x_k,y_1,y_2\}$. 

    We consider the $(k+1)$-hyperedges in $\left(\{x_1,\dots,x_{k},y_1,y_2\};{R'}^{(k+1)}_{\restr\{x_1,\dots,x_k,y_1,y_2\}}\right)$. These are associated with $(k+1)$-element subsets of $\{x_1,\dots,x_{k},y_1,y_2\}$, in the following way:
	\begin{itemize}
		\item Let $n_1$ be the number of $k$-hyperedges on the set $\{x_1,\dots,x_{k},y_1\}$, without the hyperedge $\{x_1,\dots,x_k\}$. 
		\item Let $n_2$ be the number of $k$-hyperedges on the set $\{x_1,\dots,x_{k},y_2\}$, without the hyperedge $\{x_1,\dots,x_k\}$. 
		\item Let $n_3^I$ be the number of $k$-hyperedges on the subsets $\{x_{j}:j\in I\}\cup\{y_1,y_2\}$, for $I\in\binom{[k]}{k-1}$.
	\end{itemize}
    Then, since, by definition, to check if a $(k+1)$-subset forms an edge, we simply need to check it's parity mod 2, we have that the total number of $(k+1)$-hyperedges in $\left(\{x_1,\dots,x_{k},y_1,y_2\};{R'}^{(k+1)}_{\restr\{x_1,\dots,x_k,y_1,y_2\}}\right)$ is precisely:
    \[
        n_1\pmod{2}+n_2\pmod{2}+\sum_{I\in\binom{[k-1]}{k}}\left(n_3^I\pmod{2}\right).
    \]	
    In particular, by our inductive hypothesis, we have that:
    \[
	n_1 + n_2 + \sum_{I\in\binom{[k-1]}{k}}n_3^I\equiv k \pmod{2}, 
    \]	
    Now, after adding the hyperedge $\{x_1,\dots,x_{k}\}$, the number of hyperedges will be:
    \[
        (n_1+1) + (n_2+1) + \sum_{I\in\binom{[k-1]}{k}} n_3^I\equiv n_1 + n_2 + \sum_{I\in\binom{[k-1]}{k}} n_3^I\equiv k \pmod{2},
    \]
    and thus the result follows.
\end{proof}

\begin{proposition}\label{prop:k-graphs-fraisse}
    For all $k\in\Nbb_{\geq 2}$ the class $K(\Hcal_k)$ is a Fra\"iss\'e class.
\end{proposition}

\begin{proof}
    We start with the following easy claim:
    \begin{claim}
        Let $H_i = (V_i,R_i)$, for $i\in[2]$ be $k$-hypergraphs such that $H_1$ is an induced $k$-subhypergraph of $H_2$ then $K(H_1)$ is an induced $(k+1)$-subhypergraph of $K(H_2)$.
    \end{claim}

    \begin{claimproof}
        We need to show that for all $x_1,\dots,x_{k+1}\in V_1$ we have that $K(H_1)\vDash R^{(k+1)}(x_1,\dots,x_{k+1})$ if, and only if, $K(H_2)\vDash R^{(k+1)}(x_1,\dots,x_{k+1})$. Since $H_1$ is an induced subhypergraph of $H_2$, by definition we have that $H_1\vDash R(x_{i_1},\dots,x_{i_k})$ if, and only if $H_2\vDash R(x_{i_1},\dots,x_{i_k})$. Hence:
        \[
        \Scale[0.9]{
        \begin{aligned}
        \Big\{\left\{i_1,\dots,i_k\right\}\in\binom{[k+1]}{k}:&\{x_{i_1},\dots,x_{i_k}\}\in R(H_1)\Big\} \\
        &= 	\left\{\left\{i_1,\dots,i_k\right\}\in\binom{[k+1]}{k}:\{x_{i_1},\dots,x_{i_k}\}\in R(H_2)\right\},
	\end{aligned}
        }
        \]
        and the result follows.
    \end{claimproof}

    By \Cref{prop:kay-graphs-parity}, it suffices to show that for all $k\geq 2$, the class $P^{(k)}$ is a Fra\"iss\'e class. We only show that $P^{(k)}$ has AP, since the argument for JEP is be similar. Let $K_i=(V_i,S_i)\in P^{(k)}$ be finite \emph{Kay}-graphs and suppose that we have embeddings $f_1:K_0\into K_1$ and $f_2:K_0\into K_2$. Fix a point $\star\in V(H_0)$ and let $H_i = (V_i;R_i)$, where $R_i$ is defined as in \Cref{prop:kay-graphs-parity}. By the proof of that proposition, we have that $K(H_i) = K_i$. Now, since the class of all $k$-hypergraphs has the free amalgamation property, let $H\in\Hcal_k$ be the free amalgam of $H_1$ and $H_2$ over $H_0$. By the previous claim, we have that $K_i = K(H_i)$ embed into $K(H)$, and these embeddings preserve the inclusions $K_0\subseteq K_i$, for $i\in[2]$. Since $K(H)$ is a \emph{Kay}-graph, we have that $K(H)\in P^{(k)}$, and the result follows.  
\end{proof}

\section{Approximate Ramsey properties for \emph{Kay}-graphs}\label{sec:approximate}

We recall some of the machinery from \cite{Zucker2015}. To keep notation consistent, let $\L\subseteq\L^\star$ be relational languages and $\Ccal$, $\Ccal^\star$ classes of finite structures in $\L$ and $\L^\star$, respectively, such that $\Ccal^\star$ is an expansion of $\Ccal$ (i.e. for all structures in $A^\star\in \Ccal^\star$ there is some structure $A\in\Ccal$ such that $A = A^\star_{\restr\L}$). We will focus on pairs of classes of structures of the form $(\Ccal^\star,\Ccal)$. 

\begin{definition}
    We say that $\Ccal^\star$ is a \emph{reasonable expansion} of $\Ccal$ if for any $A, B\in\Ccal$, embedding $f : A \into B$, and expansion $A^\star\in\Ccal^\star$ of $A$, then there is an expansion $B^\star\in\Ccal^\star$ of $B$ such that $f: A^\star\into B^\star$ is an embedding.
\end{definition}

\begin{fact}{\cite[Proposition~5.2]{KechrisPestovTodorcevic_2005}}\label{fact:fraisse-reasonable}
    If $\Ccal$ and $\Ccal^\star$ are Fra\"iss\'e classes then $\Ccal^\star$ is a reasonable expansion of $\Ccal$ if, and only if $\Flim(\Ccal^\star)_{\restr{\L}} = \Flim(\Ccal)$.
\end{fact}

\begin{definition}
    We say that $\Ccal^\star$ is a \emph{precompact expansion} of $\Ccal$ if for all $A\in\Ccal$ the set $\{A^\star\in\Ccal^\star:A = A^\star_{\restr\L}\}$ is finite.
\end{definition}

\begin{definition}
    We say that $\Ccal^\star$ has the \emph{expansion property} for $\Ccal$ if for any $A^\star\in\Ccal^\star$ there is some $B \in\Ccal$ such that for any expansion $B^\star\in\Ccal^\star$ of $B$, there is an embedding $f: A^\star\into B^\star$.
\end{definition}

We call $(\Ccal^\star,\Ccal)$ an \emph{excellent pair} if $\Ccal^\star$ is a reasonable precompact expansion of $\Ccal$ with the expansion property (for $\Ccal$) and $\Ccal^\star$ has ERP.

The result from \cite{Zucker2015} we will need to make use of is the following:

\begin{fact}{\cite[Proposition~5.8]{Zucker2015}}\label{fact:rd-in-excellent-pairs}
    Let $(\Ccal^\star,\Ccal)$ be an excellent pair. Then every $A\in\Ccal$ has finite Ramsey degree. In particular, the Ramsey degree of $A$ is equal to the number of expansions of $A$ in $\Ccal^\star$.
\end{fact} 

Let $\Hcal_k^o$ denote the class of all finite ordered $k$-hypergraphs. It is a well-known fact (due to \cite{AH_1978}, but also proved in \cite{NesetrilRodl1989}) that $\Hcal_{k}^o$ has ERP. Given an ordered $k$-hypergraph $H = (V,R,\leq)$, we will write $E(H) = \left(V;R,R^{(k+1)},\leq\right)$, for the expansion of $H$ by the \emph{Kay}-graph it induces. We will write $E(\Hcal_{k}^o)$ for the class $\{E(H):H\in\Hcal_k^o\}$. Since, for all $H\in\Hcal_k^o$ we have that $E(H)$ is quantifier-freely definable from $H$, the fact that $\Hcal_{k}^o$ has ERP implies that $E(\Hcal_k^o)$ also has ERP. We will write $\L = \{S,\leq\}$ for the language of $K(\Hcal_k^o)$ and $\L^\star = \{R,S,\leq\}$ for the language of $E(\Hcal_o^k)$.

\begin{remark}\label{rmk:complement}
    Let $A = (V;R,S,\leq) \in E(\Hcal_k^o)$. We will denote by $\overline{A}$ the \emph{complement} of $A$, that is $\overline{A} = (V;\binom{V}{k}\setminus R,S',\leq)$, where $S' = (\binom{V}{k}\setminus R)^{(k+1)}$. It is easy to see, by unfolding the definitions that $S=S'$ if, and only if, $k$ is even, and $S' = \binom{V}{k+1}\setminus S$ if, and only if, $k$ is odd.
\end{remark}
\begin{proposition}\label{prop:excellent-pair}
    The pair $\left(E(\Hcal_o^k),K(\Hcal_o^k)\right)$ is excellent.
\end{proposition}

\begin{proof}
    We know that $E(\Hcal_o^k)$ has ERP, so we need to show that the expansion $E(\Hcal_o^k)$ of $K(\Hcal_o^k)$ is reasonable, precompact and that the pair $\left(E(\Hcal_o^k),K(\Hcal_o^k)\right)$ has the expansion property. 
    
    Reasonability follows immediately from \Cref{prop:k-graphs-fraisse} and \Cref{fact:fraisse-reasonable}. To see that the expansion is precompact, observe that for all $A\in K(\Hcal_o^k)$ the number of $k$-hyperedge relations that we could expand $A$ by (not necessarily so that the resulting expansion is in $E(\Hcal_o^k)$) is precisely $2^{|\binom{V(A)}{k}|}$, and hence $\{A^\star\in E(\Hcal_o^k):A = A^\star_{\restr\L}\}$ is finite.

    It remains to show that $\left(E(\Hcal_o^k),K(\Hcal_o^k)\right)$ has the expansion property. Let $A^\star\in E(\Hcal_o^k)$. We need to show that there is some $B \in K(\Hcal_o^k)$ such that for any expansion $B^\star\in E(\Hcal_o^k)$ of $B$, there is an embedding $f: A^\star\into B^\star$. To this end, let $\tilde{A} := A^\star\sqcup \overline{A^\star}$, where $\overline{A^\star}$ is the complement of $A^\star$, as discussed in \Cref{rmk:complement}.

    Now, since $E(\Hcal_k^o)$ has ERP, let $D\in E(\Hcal_k^o)$ be such that $D\to\big(\tilde{A}\big)^{([k-1],\leq)}_2$ and take $D^\star = D\sqcup\{\star\}$, where $\star$ is some new vertex, not in $V(D)$, with hyperedge relation given by $\Scale[0.9]{R(D)\sqcup\left\{\{\star\}\sqcup\{x_1,\dots,x_{k-1}\}:\{x_1,\dots,x_{k-1}\}\in\binom{V(D)}{k-1}\right\}}$. Observe that by construction of $D^\star$, for any $x_1,\dots,x_k\in V(D)$ we have that:
    \begin{equation}\label{eq:kay-graph-in-d-star}
        D^\star\vDash R^{(k+1)}(x_1,\dots,x_{k},\star)\text{ if, and only if, } D^\star\vDash R(x_1,\dots,x_k).
    \end{equation}
		
    To see this, note that $D^\star\vDash R^{(k+1)}(x_1,\dots,x_{k},\star)$ if, and only if the cardinality of the set $|\left\{\left\{i_1,\dots,i_k\right\}\in\binom{[k+1]}{k}:\{x_{i_1},\dots,x_{i_k}\}\in R\right\}|$, where $x_{k+1}=\star$ is congruent to $k+1$ modulo $2$, but:
    \[
    \Scale[0.85]{
    \begin{aligned}
        |\bigg\{\left\{i_1,\dots,i_k\right\}\in\binom{[k+1]}{k}&:\{x_{i_1},\dots,x_{i_k}\}\in R\bigg\}| \\
        &= |\left\{\left\{i_1,\dots,i_{k-1}\right\}\in\binom{[k]}{k-1}:\{x_{i_1},\dots,x_{i_{k-1}},\star\}\in R\right\}| + \mathsf{t}
    \end{aligned}
    }
    \]
    where $\mathsf{t}=1$ if, and only if $D^\star\vDash R(x_1,\dots,x_k)$. Since in $D^\star$, the vertex $\star$ is connected with all $(k-1)$-element subsets, we have that
    \[
        |\left\{\left\{i_1,\dots,i_{k-1}\right\}\in\binom{[k]}{k-1}:\{x_{i_1},\dots,x_{i_{k-1}},\star\}\in R\right\}| = \binom{k}{k-1} = k.
    \]
    Thus $D^\star\vDash R^{(k+1)}(x_1,\dots,x_{k},\star)$ if, and only if $\mathsf{t}=1$, as claimed. 
		
    We claim that $B = D^\star_{\restr\L}$ witnesses the expansion property for $A$. Indeed, let us write $B = (V(D)\sqcup\{\star\},S,\leq)$, and let $B^\star = (V(D)\sqcup\{\star\};R^\star,S,\leq)\in E(\Hcal_o^k)$ be some expansion of $B$.  Observe that the $(k+1)$-hyperedge relation $S$ of $B^\star$ is precisely the \emph{Kay}-graph relation of $D^\star$. In particular, for all $x_1,\dots,x_k\in V(D)$ we have that:
    \begin{equation}\label{eq:d-star-b-star}
        D^\star\vDash R^{(k+1)}(x_1,\dots,x_k,\star)\text{ if, and only if, } B^\star\vDash S(x_1,\dots,x_k,\star).
    \end{equation}
    We need to show that there is an embedding $f:A^\star\into B^\star$. Consider the following colouring of the $(k-1)$-element substructures of $D\subseteq D^\star$ (these are just linear orders), which depends on the expansion $B^\star$ of $B$:
    \[
    \begin{aligned}
        c:\binom{D}{([k-1],\leq)}   &\to \{1,2\}\\
        (d_1\leq\cdots\leq d_{k-1})	&\mapsto 
                        \begin{cases}
                                    1 &\text{ if } \{d_1,\dots,d_{k-1},\star\}\in R^\star \\
								2 &\text{ otherwise}.
					\end{cases}
    \end{aligned}
    \]
    Since $D\to\big(\tilde{A}\big)^{([k-1],\leq)}_2$, there is some $\tilde{A}'\in\binom{D^\star\setminus\{\star\}}{\tilde{A}}$ such that $c_{\restr{\tilde{A'}}}$ is constant. By construction of $\tilde{A}$ we know that $\tilde{A}' \iso A^\star \sqcup \overline{A^\star}$. To simplify notation, let us write $A^\star \sqcup \overline{A^\star}$ for $\tilde{A}'\subseteq D^\star$. We consider the two possible outcomes separately:

    \emph{Case 1.} $\mathsf{im}(c_{\restr{\tilde{A'}}}) = \{1\}$. In this case we claim that for all $\{x_1,\dots,x_{k}\}\in \binom{A^\star}{k}$ we have that:
    \[
        D^\star\vDash R(x_1,\dots,x_k)\text{ if, and only if, } B^\star\vDash R^\star(x_1,\dots,x_k).
    \]
    By \Cref{eq:kay-graph-in-d-star,eq:d-star-b-star} we have that $D^\star\vDash R(x_1,\dots,x_k)\text{ if, and only if, } B^\star\vDash S(x_1,\dots,x_k,\star)$, thus, it suffices to show that:
    \[
        B^\star\vDash S(x_1,\dots,x_k,\star)\text{ if, and only if } B^\star\vDash R^\star(x_1,\dots,x_k)
    \]
    By definition, after putting $x_{k+1}=\star$ we have that $B^\star\vDash S(x_1,\dots,x_k,\star)$ if, and only if, $|\left\{\left\{i_1,\dots,i_k\right\}\in\binom{[k+1]}{k}:\{x_{i_1},\dots,x_{i_k}\}\in R^\star\right\}$ is congruent to $k+1$ modulo $2$, but:
    \[
    \Scale[0.85]{
    \begin{aligned}
        |\bigg\{\left\{i_1,\dots,i_k\right\}\in\binom{[k+1]}{k}&:\{x_{i_1},\dots,x_{i_k}\}\in R^\star\bigg\}| \\
        &= |\left\{\left\{i_1,\dots,i_{k-1}\right\}\in\binom{[k]}{k-1}:\{x_{i_1},\dots,x_{i_{k-1}},\star\}\in R^\star\right\}| + \mathsf{t},
    \end{aligned}
    }
    \]
    where $\mathsf{t}=1$ if, and only if, $B^\star\vDash R^\star(x_1,\dots,x_k)$. Since $\mathsf{im}(c_{\restr{\tilde{A'}}}) = \{1\}$ we have that
    \begin{equation}\label{eq:case-i}
    \Scale[0.9]{
        |\left\{\left\{i_1,\dots,i_{k-1}\right\}\in\binom{[k]}{k-1}:\{x_{i_1},\dots,x_{i_{k-1}},\star\}\in R^\star\right\}| = \binom{k}{k-1}= k
    }
    \end{equation}
    So, we conclude that:
    \[
        B^\star\vDash S(x_1,\dots,x_k,\star)\text{ if, and only if } B^\star\vDash R^\star(x_1,\dots,x_k),
    \]
    as claimed.
			
   \emph{Case 2.} $\mathsf{im}(c_{\restr{\tilde{A'}}}) = \{2\}$. The argument will now vary slightly, depending on whether $k$ is even or odd. We consider the two cases separately: 
			
   \emph{Assume $k$ is even}. In this case we again claim that for all $\{x_1,\dots,x_{k}\}\in \binom{A^\star}{k}$ we have that:
    \[
        D^\star\vDash R(x_1,\dots,x_k)\text{ if, and only if, } B^\star\vDash R^\star(x_1,\dots,x_k).
    \]
    The argument here is identical to the one given in Case 1, above, with the only difference that instead of \Cref{eq:case-i}, we have that:
    \[
        |\left\{\left\{i_1,\dots,i_{k-1}\right\}\in\binom{[k]}{k-1}:\{x_{i_1},\dots,x_{i_{k-1}},\star\}\in R^\star\right\}| = 0,
    \]
   since $\mathsf{im}(c_{\restr\tilde{A'}}) = \{2\}$. Again, this forces $\mathsf{t}=1$, and the result follows as before.
								   		
    \emph{Assume $k$ is odd}. In this case we claim that for all $\{x_1,\dots,x_{k}\}\in \binom{\overline{A^\star}}{k}$ we have that:
    \[
        D^\star\vDash \lnot R(x_1,\dots,x_k)\text{ if, and only if, } B^\star\vDash R^\star(x_1,\dots,x_k).
    \]
    Again, by \Cref{eq:kay-graph-in-d-star,eq:d-star-b-star} it suffices to show that:
    \[
        B^\star\vDash S(x_1,\dots,x_k,\star)\text{ if, and only if } B^\star\vDash \lnot R^\star(x_1,\dots,x_k)
    \]
    Arguing as in Case 1, $B^\star\vDash S(x_1,\dots,x_k,\star)$ if, and only if,
    \[
        |\left\{\left\{i_1,\dots,i_{k-1}\right\}\in\binom{[k]}{k-1}:\{x_{i_1},\dots,x_{i_{k-1}},\star\}\in R^\star\right\}| + \mathsf{t} \equiv 0\pmod{2}
    \]
    where $\mathsf{t}=1$ if, and only if $B^\star\vDash R^\star(x_1,\dots,x_k)$. Since $\mathsf{im}(c_{\restr{\tilde{A'}}}) = \{1\}$, just like the even case, we have that
    \[
        |\left\{\left\{i_1,\dots,i_{k-1}\right\}\in\binom{[k]}{k-1}:\{x_{i_1},\dots,x_{i_{k-1}},\star\}\in R^\star\right\}| = 0
    \]
    and hence $B^\star\vDash S(x_1,\dots,x_k,\star)$ holds if, and only if, $\mathsf{t}=0$. So, we conclude that:
    \[
        B^\star\vDash S(x_1,\dots,x_k,\star)\text{ if, and only if } B^\star\vDash \lnot R^\star(x_1,\dots,x_k),
    \]
    as claimed.
\end{proof}

We almost immediately conclude the following:

\begin{corollary}\label{cor:main}
    Let $k\in\Nbb_{\geq 3}$. Then $K(\Hcal_o^k)$ has $(k-1)$-ERP but not $k$-ERP
\end{corollary}
\begin{proof}
    By \Cref{prop:excellent-pair}, $\left(E(\Hcal_o^k),K(\Hcal_o^k)\right)$ is an excellent pair, so applying \Cref{fact:rd-in-excellent-pairs}, for all $A=(V,S,\leq)\in K(\Hcal_o^k)$ we have that the Ramsey degree of $A$ is equal to the number of expansions of $A$ in $E(\Hcal_o^k)$. Observe that if $|V(A)|\leq k$ then $S(A)=\emptyset$, since hypergraphs are assumed to be uniform. Now, if $|V(A)|\leq k-1$, then $A$ has a unique (trivial) expansion in $E(\Hcal_o^k)$, namely $A^\star = (V;\emptyset,\emptyset,\leq)$. On the other hand, if $|V(A)|=k$, then $A$ has two expansions in $E(\Hcal_o^k)$, namely $A_1^\star = (V;\{V\},\emptyset,\leq)$ and $A_2^\star = (V;\emptyset,\emptyset,\leq)$.
\end{proof}

\begin{remark}\label{rmk:structural}
    Since the classes in \Cref{cor:main}, above, consist of ordered structures, the embedding Ramsey property and the structural Ramsey property coincide. Thus, if one considers finitary approximations to the \emph{structural} Ramsey property, instead of the embedding Ramsey property, \Cref{cor:main} is still applicable.
\end{remark}

\section{Orders in Ramsey classes}\label{sec:elem-ord}

In this \namecref{sec:elem-ord} we prove \Cref{thmx:order}. Let us start by formally defining the notions we require.

\begin{definition}[$(\infty,0)$-orderability]\label{def:orderability}
    Let $\C$ be a class of $\L$-structures. We say that $\C$ is \emph{$(\infty,0)$-orderable} if there is a (possibly infinite) Boolean combination of atomic formulas $\Phi(x,y):=\bigvee_{i\in I}\bigwedge_{j\in J_i} \vphi_{j_i}^{(-1)^{n_{j_i}}}(x,y)$, such that $\Phi$ is a linear order for every structure in $\C$.
\end{definition}

The term $(\infty,0)$-orderable is meant to evoke a connection with infinitary logic, which is not crucial to this paper. For more information, we refer the reader to \cite{MP}. 

\begin{theorem}\label{thm:2-ERP-implies-order}
    Let $\C$ be a class of finite structures. If $\C$ has $2$-ERP, then $\C$ is $(\infty,0)$-orderable. 
\end{theorem}

\begin{proof}
    By using one of the techniques of \cite{MP} (the \emph{quantifier-free type Morleyisation}), we may assume that in all structures in $\C$ all quantifier-free types that are realised are isolated. Let $\cQ$ be the set of quantifier-free $\cL$-formulas in two variables isolating every quantifier-free type in two variables which is realised in some member of $\C$.
        
    Abusing notation, let $\C/\cong$ be a set of representatives of structures in $\C$ under the equivalence relation of isomorphism. Let 
    \[
        \cL':= \cL\cup\{<\}\cup\bigcup_{A\in \C/\cong} A,
    \]
    where $<$ is a binary relation symbol $2$ and each $a\in \bigcup_{A\in \C/\cong} A$ is a constant symbol and let $T$ be the following $\cL'$-theory:
    \[ 
        \bigcup_{A\in \C/\cong}\Diag(A)\cup \{\left(Q(x_1,y_1)\land Q(x_2,y_2)\right)\rightarrow \left( x_1<y_1\leftrightarrow x_2<y_2 \right): Q\in\cQ\}\cup \mathsf{LO} 
    \]
    where $\mathsf{LO}$ expresses that $<$ is a linear order. Then the \namecref{thm:2-ERP-implies-order} follows if $T$ is finitely satisfiable. By JEP, it suffices to show that: 
    \[ 
        \Diag(B)\cup \{\left(Q(x_1,y_1)\land Q(x_2,y_2)\right)\rightarrow \left( x_1<y_1\leftrightarrow x_2<y_2 \right): Q\in\cQ\}\cup \mathsf{LO} 
    \]
    is satisfiable for every $B\in \C$.

    Given $B\in \C$, let $A_1,\dots, A_n$ be the set of all substructures of $B$ of size $2$, and for each $i\leq n$ fix an enumeration $\{a_i^1,a_i^2\}$ of $A_i$.
        
    Let $C\in \C$ be as promised from \Cref{k-ERP-many}. Now let $\widehat{C}$ be an expansion of $C$ to $\cL\cup\{<\}$ such that $<$ is a linear order and for every $i\leq n$, define a colouring:
    \[
        \begin{aligned}
            \chi_i: \binom{C}{A_i} & \to \{0,1\}\\
            \{a_i^1,a_i^2\}        & \mapsto 
                        \begin{cases}
                            0 &\text{ if }a_i^0<a_i^1,\\
                            1 &\text{ if }a_i^1<a_i^0,
                        \end{cases}
        \end{aligned}
    \]
    where the ordering is taken in $\widehat{C}$. Let $B^\prime\in\binom{\widehat{C}}{B}$ be such that $\chi_i$ is constant on $\binom{B^\prime}{A_i}$, for all $i\leq n$. The resulting $B'$, with the ordering $<$ inherited from $\widehat{C}$ satisfies
    \[ 
        \Diag(B)\cup \{\left(Q(x_1,y_1)\land Q(x_2,y_2)\right)\rightarrow \left( x_1<y_1\leftrightarrow x_2<y_2 \right): Q\in\cQ\}\cup \mathsf{LO}.
    \]
    Indeed, suppose that we are given $b_1,b_2\in B^\prime$, then, for some $i$ we have that $\{b_1,b_2\} = A_i$. Then for all $c_1,c_2\in B^\prime$, if we have that $B^\prime\vDash Q(b_1,b_2)\land Q(c_1,c_2)$, for all $Q\in\cQ$ then $\{c_1,c_2\}\iso A_i$, hence $\chi_i(\{b_1,b_2\}) = \chi_i(\{c_1,c_2\})$. But since ${\chi_i}_{\restr{\binom{B^\prime}{A_i}}}$ is constant, we must have that $b_1,b_2$ and $c_1,c_2$ have the same order. 
\end{proof}
    
\begin{proposition}\label{qford rigid}
    Let $\C$ be a class of finite $\L$-structures. If $\C$ is $(\infty,0)$-orderable, then every $A\in \C$ is rigid.
\end{proposition}

\begin{proof}
    Let $A\in \C$ and let $\widehat{A}$ be the expansion of $A$ to $\cL\cup\{\preceq\}$ where $\preceq$ is a binary relation symbol and for all $a,b\in A$ we have that $a\preceq b$ if, and only if $\Phi(a,b)$, where $\Phi$ is a Boolean combination of atomic and negated atomic formulas linearly ordering all structures in $\C$. Then $\aut(A)=\aut(\widehat{A})$, but since $\widehat{A}$ is a finite linearly ordered structure, it is rigid, and hence $A$ must also be rigid.
\end{proof}

We conclude our paper with an example of a class of finite rigid structures which is not $(\infty,0)$-orderable. This example is due to P. Cameron and was communicated to us by D. Macpherson: 

\begin{example}\label{example:rigid not ordered finite language}
    Let $\L := \{R,C\}$, where $R$ is a binary relation symbol and $C$ is a ternary relation symbol. Let $\C$ be the class of finite $\L$ structures where $R$ is a tournament and $C$ is a $C$-relation derived from a binary tree (see \cite{Cameron_1990} for the relevant definitions). It is an easy observation that all elements of $\C$ are rigid (since the automorphism groups of finite structures with $C$-relations derived from binary trees are $2$-groups and the automorphism groups of tournaments have odd order). It is also fairly clear that $\C$ is an amalgamation class. If $\cM = \Flim(\Ccal)$ admitted an $\aut(\cM)$-invariant linear order $\preceq$, then by ultrahomogeneity and antisymmetry, $\preceq$ would be a union of quantifier-free types in two variables $x,y$. Therefore $\preceq$ would be definable using only equality and $R$. So if $\cM_2:=\cM_{\restr\{R\}}$, then $\preceq$ would be $\aut(\cM_2)$-invariant. Notice that $\cM_2$ is the random tournament, which doesn't admit an invariant linear order (e.g., a $3$-cycle isn't rigid).
\end{example}

\subsection*{Acknowledgements} The authors would like to thank D. Macpherson for his helpful suggestions and for pointing out \Cref{example:rigid not ordered finite language}. The authors would also like to thank Instytut Matematyczny, Uniwersytet Wrocławski for funding and hosting the second author in June 2023, when part of the research in this paper was conducted.

\bibliographystyle{alpha}
\bibliography{bibliography}

\end{document}